\documentclass[10pt]{amsart}
\usepackage{amssymb}
\usepackage{amscd}
\usepackage[all]{xy}

\numberwithin{equation}{section}

\def\today{\number\day\space\ifcase\month\or   January\or February\or
   March\or April\or May\or June\or   July\or August\or September\or
   October\or November\or December\fi\   \number\year}

\theoremstyle{definition}
\newtheorem{thm}{Theorem}[section]
\newtheorem{lem}[thm]{Lemma}
\newtheorem{prp}[thm]{Proposition}

\newtheorem{cor}[thm]{Corollary}

\newtheorem{cnv}[thm]{Convention}
\newtheorem{rmk}[thm]{Remark}

\newcommand{\beq}{\begin{equation}}
\newcommand{\eeq}{\end{equation}}
\newcommand{\beqr}{\begin{eqnarray*}}
\newcommand{\eeqr}{\end{eqnarray*}}
\newcommand{\bal}{\begin{align*}}
\newcommand{\eal}{\end{align*}}
\newcommand{\bei}{\begin{itemize}}
\newcommand{\eei}{\end{itemize}}

\newcommand{\af}{\alpha}

\newcommand{\dt}{\delta}
\newcommand{\ep}{\varepsilon}
\newcommand{\zt}{\zeta}
\newcommand{\et}{\eta}
\newcommand{\ch}{\chi}

\newcommand{\ld}{\lambda}

\newcommand{\rh}{\rho}

\newcommand{\Q}{{\mathbb{Q}}}
\newcommand{\Z}{{\mathbb{Z}}}
\newcommand{\R}{{\mathbb{R}}}
\newcommand{\C}{{\mathbb{C}}}
\newcommand{\N}{{\mathbb{Z}}_{> 0}}
\newcommand{\Nz}{{\mathbb{Z}}_{\geq 0}}

\newcommand{\id}{{\operatorname{id}}}

\newcommand{\sint}{{\operatorname{int}}}

\newcommand{\spec}{{\operatorname{sp}}}

\newcommand{\supp}{{\operatorname{supp}}}
\newcommand{\rank}{{\operatorname{rank}}}

\newcommand{\rc}{{\operatorname{rc}}}

\newcommand{\dirlim}{\varinjlim}

\newcommand{\andeqn}{\qquad {\mbox{and}} \qquad}

\newcommand{\ifo}{if and only if}

\newcommand{\ca}{C*-algebra}
\newcommand{\uca}{unital C*-algebra}

\newcommand{\pj}{projection}

\newcommand{\mvnt}{Murray-von Neumann equivalent}

\newcommand{\ct}{continuous}
\newcommand{\cfn}{continuous function}

\newcommand{\cms}{compact metrizable space}

\newcommand{\chs}{compact Hausdorff space}

\renewcommand{\S}{\subset}

\newcommand{\I}{\infty}
\newcommand{\E}{\varnothing}

\title{The radius of comparison of $C (X)$}

\author{N.~Christopher Phillips}

\date{14~September 2023}

\address{Department of Mathematics, University  of Oregon,
       Eugene OR 97403-1222, USA.}

\thanks{This material is based upon work supported by
  the Simons Foundation Collaboration Grant for Mathematicians \#587103
  and by the US National Science Foundation under Grant DMS-2055771.}

\begin{document}

\begin{abstract}
Let $X$ be a compact Hausdorff space.
Then the radius of comparison ${\mathrm{rc}} ( C (X))$
is related to the covering dimension ${\mathrm{dim}} (X)$ by
\[
{\mathrm{rc}} ( C (X)) \geq \frac{{\mathrm{dim}} (X) - 7}{2}.
\]
Except for the additive constant,
this improves a result of Elliott and Niu, who proved that
if $X$ is metrizable then
\[
{\mathrm{rc}} (C (X)) \geq \frac{{\mathrm{dim}}_{\Q} (X) - 4}{2}.
\]
There are compact metric spaces~$X$ for which the estimate of
Elliott and Niu gives no information, but for which
${\mathrm{rc}} ( C (X))$ is infinite or has arbitrarily large finite values.
\end{abstract}

\maketitle

\indent

\section{Introduction}\label{Sec_1801_Sec1}

\indent
The radius of comparison ${\mathrm{rc}} (A)$ of a \uca~$A$
was introduced in~\cite{Tms1}
to distinguish counterexamples to the Elliott classification program
in the absence of Jiang-Su stability.
Since AH~algebras are a rich source of examples of simple \ca{s},
both nonclassifiable (like the ones in~\cite{Tms1}) and classifiable,
and since the radius of comparison is sometimes well behaved in direct limits,
in particular in AH~systems
(in~\cite{BRTTW}, see Proposition 3.2.4(iii) and Proposition 3.2.3),
it is of interest to compute $\rc (C (X, M_n))$ for $n \in \N$
and a \chs~$X$.
Since $\rc (M_n (A)) = \frac{1}{n} \rc (A)$,
this reduces to the computation of ${\mathrm{rc}} ( C (X))$.

For further motivation,
let $\dim (X)$ denote the covering dimension of the topological space~$X$.
(See Definition~1.1 in Chapter~3 of~\cite{Prs}.)
It has long been known that
$\rc (C (X)) \leq \frac{1}{2} \dim (X)$.
(See (2.1) in~\cite{EllNiu2}, or Subsection~4.1 in~\cite{BRTTW},
for a slightly more precise result;
these are not the original sources.
The factor $\frac{1}{2}$ arises from the use of complex scalars.)
Moreover, $\rc (A)$ behaves somewhat like a noncommutative generalization
of a dimension for topological spaces.
This inequality therefore suggests that, as happens for other
noncommutative dimensions such as real rank, stable rank,
and decomposition rank, $\rc (C (X))$ should be related to $\dim (X)$.

For a \cms~$X$ and an abelian group~$G$, let $\dim_G (X)$ be the
cohomological dimension of $X$ with coefficients~$G$,
as in Section~1 of~\cite{Drnv}.
In~\cite{EllNiu2}, Elliott and Niu proved that
\begin{equation}\label{Eq_2511_EN}
\rc (C (X)) \geq \frac{{\mathrm{dim}}_{\Q} (X) - 4}{2}.
\end{equation}
(They give a slightly better estimate when ${\mathrm{dim}}_{\Q} (X)$
is odd, but that estimate can be recovered from~(\ref{Eq_2511_EN})
by rounding up, because $\rc (C (X)) \in \Nz \cup \{ \I \}$
by Corollary~2.4 of~\cite{EllNiu2};
also see Lemma~\ref{L_1725_Dfn_rcCX} below.)
While satisfactory for spaces such as finite complexes,
this result leaves open the question of whether, in general,
$\rc (C (X))$ is related to $\dim (X)$, $\dim_{\Q} (X)$,
something else (such as $\dim_{\Z} (X)$),
or some new dimension for topological spaces.
We prove (Corollary~\ref{C_2705_Cor} below) that for any \cms~$X$, one has
\begin{equation}\label{Eq_2511_cd}
\rc (C (X)) \geq \frac{{\mathrm{dim}} (X) - 7}{2},
\end{equation}
showing that the covering dimension is at least nearly the
right commutative dimension.
(We get a slightly better bound for some values of $\dim (X)$;
see Theorem~\ref{T_2705_Ests} below.)
On the other hand, the additive constant in~(\ref{Eq_2511_cd})
is worse than the one in~(\ref{Eq_2511_EN}),
and we don't know how to improve it in general.

We illustrate the differences between ${\mathrm{dim}} (X)$
and ${\mathrm{dim}}_{\Q} (X)$.
To begin, recall from Example 1.3(1) of~\cite{Drnv}
that for any \cms~$X$ and any abelian group~$G$, one has
\[
{\mathrm{dim}}_{G} (X)
 \leq {\mathrm{dim}}_{\Z} (X)
 \leq {\mathrm{dim}} (X).
\]
Theorem~7.1 of~\cite{Drnv} gives a \cms~$X$
such that ${\mathrm{dim}} (X) = \I$ and ${\mathrm{dim}}_{\Z} (X) \leq 3$.
So ${\mathrm{dim}}_{\Q} (X) \leq 3$.
The best estimate gotten from~(\ref{Eq_2511_EN})
is $\rc (C (X)) \geq 0$, but the estimate~(\ref{Eq_2511_cd})
shows that $\rc (C (X)) = \I$.
We point out that this space~$X$ contains no closed subspace~$Y$
with $3 < \dim (Y) < \I$.
Indeed, if $Y \S X$ is closed and $\dim (Y) < \I$,
then, using Theorem~1.4 of~\cite{Drnv} and Corollary~1.2 of~\cite{Drnv},
we have
$\dim (Y) = {\mathrm{dim}}_{\Z} (Y) \leq {\mathrm{dim}}_{\Z} (X) \leq 3$.

Theorem~1.4 of~\cite{Drnv} states that if ${\mathrm{dim}} (X) < \I$,
then ${\mathrm{dim}}_{\Z} (X) = {\mathrm{dim}} (X)$.
However, it is still possible for ${\mathrm{dim}}_{\Q} (X)$
to be much smaller than ${\mathrm{dim}}_{\Z} (X)$.
Let $P \S \N$ be the set of primes.
For $p \in P$, approximately following the introduction to
Section~2 of~\cite{Drnv},
let $\Z_{\langle p \rangle}$ denote the localization of $\Z$ at the
prime ideal $\langle p \rangle$ generated by~$p$, and let
$\Z_{p^{\infty}} = \dirlim_n \Z / p^n \Z$,
via the maps $\Z / p^n \Z \to \Z / p^{n + 1} \Z$
which send $1 + p^n \Z$ to $p + p^{n + 1} \Z$.
For an example of what can happen,
we combine Proposition~4.1 of~\cite{Drnv} and Theorem~5.1 of~\cite{Drnv},
referring to Theorem~2.4 of~\cite{Drnv} for the statement of the
Bockstein inequalities, to see that for every function
$r \colon P \to \N \cup \{ \I \}$ there is a \cms~$X$ such that
\[
{\mathrm{dim}}_{\Q} (X) = 1,
\quad
{\mathrm{dim}}_{\Z / p \Z} (X) = {\mathrm{dim}}_{\Z_{p^{\infty}}} (X) = r (p),
\quad {\mbox{and}} \quad
{\mathrm{dim}}_{\Z_{\langle p \rangle}} (X) = r (p) + 1.
\]
It then follows from Theorem~2.1 of~\cite{Drnv}
that ${\mathrm{dim}}_{\Z} (X) = 1 + \sup_{p \in P} r (p)$.
It is not stated in~\cite{Drnv}, but one can see by examining the
proofs there that the space~$X$ gotten from the construction
also satisfies ${\mathrm{dim}} (X) = 1 + \sup_{p \in P} r (p)$.
(The key point is that the spaces constructed in Corollary~5.3 of~\cite{Drnv}
have dimension~$n$.)
For every $d \in \{ 2, 3, \ldots, \I \}$, an appropriate choice of $r$
gives a \cms~$X$ such that
${\mathrm{dim}}_{\Z} (X) = {\mathrm{dim}} (X) = d$,
but ${\mathrm{dim}}_{\Q} (X) = 1$.
The estimate~(\ref{Eq_2511_EN}) gives no information,
but the estimate~(\ref{Eq_2511_cd})
shows that $\rc (C (X)) \geq \frac{1}{2} (d - 7)$,
which can be infinite, or finite but arbitrarily large.

Elliott and Niu build witnesses to lower bounds for $\rc (C (X))$
using Chern classes and the Chern character.
This approach depends on vector bundles with nonzero classes
in $\Q \otimes K^0 (Y)$ for suitable subspaces $Y \S X$.
It requires that $X$ have closed subspaces with nonzero
cohomology in high degrees,
and therefore can't work for the space $X$ discussed above with
${\mathrm{dim}} (X) = \I$ and ${\mathrm{dim}}_{\Z} (X) \leq 3$.

Our proof uses instead the following standard characterization
of the covering dimension $\dim (X)$.
The statement in~\cite{Prs} is actually more general:
it applies to arbitrary normal topological spaces.

\begin{thm}[Theorem 2.2 in Chapter~3
 of~\cite{Prs}]\label{T_1801_Dim_exts}
Let $X$ be a \chs{} and let $d \in \Nz$.
Then $\dim (X) \leq d$ \ifo{} for every closed subset $X_0 \S X$
and every \cfn{} $\xi_0 \colon X_0 \to S^d$,
there is a \cfn{} $\xi \colon X \to S^d$ such that $\xi |_{X_0} = \xi_0$.
\end{thm}

We relate extendability of maps $X_0 \to S^{2 d - 1}$
for closed subspaces $X_0$ of $X$
to extendability of trivial rank one projections in $C (X_0, M_d)$,
and relate extendability of \pj{s} to the radius of comparison.
The proof does implicitly contain vector bundles over subspaces $X_0 \S X$,
but they are all stably trivial, so that their classes in $K^0 (X_0)$
are zero.

Most of this work was done during a visit to the
Westf\"{a}lische Wilhelms-Universit\"{a}t M\"{u}nster in July 2021,
and we are grateful to that institution for its hospitality.
The key Proposition~\ref{T_1801_Htpy} was provided by Thomas Nikolaus
during that visit.
We are also grateful to John Klein for helpful email correspondence.

\section{Extending projections}\label{Sec_2821_Sec2}

In this section, we prove, in effect, that if $r \geq \rc (C (X))$,
then for $X_0 \S X$ closed,
every rank one \pj{} $p_0 \in C (X_0, M_{r + 3})$ which is \mvnt{} to a
constant \pj{} can be
extended to a rank one \pj{} $p \in C (X, M_{r + 3})$ with the same property.
(See Remark~\ref{R_1726_Ext_pjs}.)
Without the restriction on the matrix size of~$p$,
there is always an extension, in fact,
an extension to a \pj{} in $C (X, M_{4 (r + 3)})$.
The point here is that the matrix size need not be increased.

The following somewhat nonstandard terminology is convenient.

\begin{cnv}\label{Cv_1722_Triv}
Let $X$ be a \chs, let $n \in \N$, and let $p \in C (X, M_n)$ be a \pj.
We say that $p$ is {\emph{trivial}} if $p$ is \mvnt{}
in $C (X, M_n)$ to a constant \pj.
\end{cnv}

A trivial \pj{} in $C (X, M_n)$
is \mvnt{} in $C (X, M_n)$ to any constant \pj{} with the correct rank.

\begin{lem}\label{L_1725_Extend}
Let $X$ be a \chs, and let $Y \S X$ be closed.
Let $n \in \N$, and let $p_0 \in C (Y, M_n)$ and $q \in C (X, M_n)$
be \pj{s}.
Suppose $p_0 (x) \perp q (x)$ for all $x \in Y$.
\begin{enumerate}
\item\label{Item_1725_Extend_p}
There are a closed set $Z \S X$ with $Y \S \sint (Z)$
and a \pj{} $p \in C (Z, M_n)$
such that $p |_{Y} = p_0$ and $p (x) \perp q (x)$ for all $x \in Z$.
\item\label{Item_1725_Extend_isom}
If moreover $e \in C (X, M_n)$ is a \pj,
and $s_0 \in C (Y, M_n)$
satisfies $s_0 s_0^* = p_0$ and $s_0^* s_0 = e |_Y$,
then $Z$ and~$p$ in (\ref{Item_1725_Extend_p})
may be chosen so that there is $s \in C (Z, M_n)$
satisfying $s s^* = p$, $s^* s = e |_Z$, and $s |_Y = s_0$.
\end{enumerate}
\end{lem}

\begin{proof}
For the first part,
choose a positive element $a \in C (X, M_{n})$
such that $a |_{X_0} = p_0$.
For $x \in X$ define $b (x) = [ 1 - q (x) ] a (x) [ 1 - q (x) ]$.
Then $b (x) = p_0 (x)$ for all $x \in Y$.
So there is a closed set $Z \S X$ with $Y \S \sint (Z)$
such that for all $x \in Z$,
we have $\| b (x)^2 - b (x) \| < \frac{1}{4}$.
This implies $\frac{1}{2} \not\in \spec ( b (x))$,
so the \ct{} functional calculus
$p (x) = \ch_{ ( \frac{1}{2}, \I]} (b (x))$ is defined,
and gives a \ct{} \pj{} valued function on~$Z$.
Clearly $p |_{Y} = p_0$ and $p (x) \perp q (x)$ for all $x \in Z$.

For the second part, apply~(\ref{Item_1725_Extend_p}),
calling the resulting closed set and \pj{} $T$ and~$r$.
Choose any $t \in C (X, M_n)$ such that $t |_Y = s_0$.
For $x \in T$ define $c (x) = r (x) t (x) e (x)$.
Then $c (x) = s_0 (x)$ for all $x \in Y$.
So there is a closed set $Z \S X$ with $Y \S \sint (Z) \S Z \S T$
and which is so small that for all $x \in Z$, the expression
$s (x) = c (x) [ c (x)^* c (x) ]^{- 1/2}$,
with functional calculus in $e (x) M_{n} e (x)$,
is defined and satisfies $s (x) s (x)^* = r (x)$
and $s (x)^* s (x) = e (x)$.
Take $p = r |_Z$.
\end{proof}

\begin{lem}\label{L_1723_Flip}
Let $X$ be a \chs, and let $Y_1, Y_2 \S X$ be disjoint closed sets.
Let $n \in \N$, and let $p_1, p_2 \in C (X, M_n)$
be rank one \pj{s} which are orthogonal
and are both trivial in the sense of Convention~\ref{Cv_1722_Triv}.
Then there exists a unitary $w \in (p_1 + p_2) C (X, M_n) (p_1 + p_2)$
such that $w |_{Y_1} = (p_1 + p_2) |_{Y_1}$ and the \pj{}
$p = w p_1 w^*$ is trivial and satisfies $p |_{Y_j} = p_j |_{Y_j}$
for $j = 1, 2$.
\end{lem}

The proof works just as well whenever $p_1$ and $p_2$ are trivial
projections of the same (constant) rank, not necessarily rank one.

\begin{proof}[Proof of Lemma~\ref{L_1723_Flip}]
We can replace $p_1$ and $p_2$ with \mvnt{} \pj{s},
provided the replacements are still orthogonal.
We can also replace $C (X, M_n)$ with
$(p_1 + p_2) C (X, M_n) (p_1 + p_2)$.
Therefore we may assume that $n = 2$ and that $p_1$ and~$p_2$
are the constant \pj{s} with values
$\left( \begin{smallmatrix} 1  & 0 \\ 0  &  0 \end{smallmatrix} \right)$
and
$\left( \begin{smallmatrix} 0  & 0 \\ 0  &  1 \end{smallmatrix} \right)$.
Choose a \ct{} path $\ld \mapsto u (\ld)$ of unitaries,
for $\ld \in [0, 1]$, such that $u (0) = 1$ and
$u (1)
 = \left( \begin{smallmatrix} 0 & 1 \\ 1 & 0 \end{smallmatrix} \right)$.
Choose a \cfn{} $f \colon X \to [0, 1]$ such that
$f (x) = 0$ for $x \in Y_1$ and $f (x) = 1$ for $x \in Y_2$.
Then define $w (x) = u (f (x))$ for $x \in X$.
\end{proof}

\begin{lem}\label{L_1723_Switch}
Let $X$ be a \chs, and let $Y \S X$ be closed.
Let $n \in \N$, and let $p_0 \in C (Y, M_n)$ be a rank one \pj{}
which is trivial in the sense of Convention~\ref{Cv_1722_Triv}.
Further let $q \in C (Y, M_n)$ be a rank one trivial \pj{}
such that $q |_Y \perp p_0$.
Then there exists a rank one trivial \pj{} $p \in C (Y, M_n)$
such that $p |_Y = p_0$.
\end{lem}

As for Lemma~\ref{L_1723_Flip}, the proof works for any rank in place
of~$1$.

The key point is that we are not allowed to increase~$n$.
If we embed $C (Y, M_n)$ in $C (Y, M_{4 n})$,
then in the larger algebra $p_0$ will be homotopic to a constant \pj,
and the existence of the extension will be automatic without using~$q$.

\begin{proof}[Proof of Lemma~\ref{L_1723_Switch}]
Choose any rank one constant \pj{} $e \in C (X, M_n)$,
and choose $s_0 \in C (Y, M_n)$ such that
$s_0 s_0^* = p_0$ and $s_0^* s_0 = e |_Y$.
Apply Lemma \ref{L_1725_Extend}(\ref{Item_1725_Extend_isom}),
getting $Z$, a \pj{} $p_1 \in C (Z, M_n)$ (called $p$ there)
such that $p_1 |_Y = p_0$,
and a partial isometry $v$ (called $s$ there)
such that $v v^* = p_1$ and $v^* v = e |_Z$.
Since $e |_Z$ is trivial of rank one, so is~$p_1$.
Because $Y \S \sint (Z)$ and $p_1 \perp q |_Z$,
we can apply Lemma~\ref{L_1723_Flip},
with $Z$ in place of~$X$, with $Y_1 = \partial Z$, with $Y_2 = Y$,
with $q |_Z$ in place of~$p_1$,
and with $p_1$ in place of~$p_2$,
getting a unitary $w \in (p_1 + q |_Z) C (Z, M_n) (p_1 + q |_Z)$
such that $w |_{\partial Z} = p_1 |_{\partial Z} + q |_{\partial Z}$
and $(w |_Y) (q |_Y) (w^* |_Y) = p_1 |_Y$.
Define
\[
p (x) = \begin{cases}
   q (x) & \hspace*{1em} x \in X \setminus Z
        \\
   w (x) q (x) w (x)^* & \hspace*{1em} x \in Z.
\end{cases}
\]
Then $p$ is \ct.
Since for $x \in Y$ we have $w (x) q (x) w (x)^* = p_1 (x) = p_0 (x)$,
it follows that $p |_{Y} = p_0$.

It remains to prove that $p$ is trivial.
Since $q$ is trivial, there is $s \in C (X, M_n)$ such that
$s s^* = e$ and $s^* s = q$.
Define
\[
t (x) = \begin{cases}
   s (x) & \hspace*{1em} x \in X \setminus Z
        \\
   s (x) w (x)^* & \hspace*{1em} x \in Z.
\end{cases}
\]
Then $t$ is \ct{}
because $w (x) = p_1 (x) + q (x)$ for $x \in \partial Z$.
For $x \in X \setminus Z$ we have
\[
t (x) t (x)^* = s (x) s (x)^* = e (x)
\andeqn
t (x)^* t (x) = s (x)^* s (x) = q (x) = p (x).
\]
For $x \in Z$ we have
\[
t (x) t (x)^*
 = s (x) w (x)^* w (x) s (x)^*
 = s (x) [p_1 (x) + q (x)] s (x)^*
 = s (x) s (x)^*
 = e (x)
\]
and
\[
t (x)^* t (x)
 = w (x) s (x)^* s (x) w (x)^*
 = w (x) q (x) w (x)^*
 = p (x).
\]
So $t t^* = e$ and $t^* t = p$.
Since $e$ is constant, this shows that $p$ is trivial.
\end{proof}

The proof of the following lemma in~\cite{ArPrTm} has
enough misprints that we give a full proof here.

\begin{lem}[Lemma 2.19 of~\cite{ArPrTm}]\label{L_1725_Cut}
Let $A$ be a \ca, let $a \in A_{+}$, and let $p \in A$ be a \pj.
Suppose $p \precsim a$.
Then there exist $\af \in (0, \I)$, a \pj{} $q \in A$,
and $\dt > 0$, such that $q$ is \mvnt{} to~$p$,
$q \leq \af a$, and, with
\[
g (\ld) = \begin{cases}
   0 & \hspace*{1em} 0 \leq \ld < \dt
        \\
   \dt^{-1} (\ld - \dt) & \hspace*{1em} \dt \leq \ld < 2 \dt
       \\
   1 & \hspace*{1em} 2 \dt \leq \ld,
\end{cases}
\]
we have $g (a) q = q$.
\end{lem}

\begin{proof}
Set $\ep = \frac{1}{4}$.
Use Proposition 2.17(iii) of~\cite{ArPrTm} to choose $\dt_0 > 0$
such that $(p - \ep)_{+} \precsim (a - \dt_0)_{+}$.
Then there is $v \in A$ such that
\[
\bigl\| v (a - \dt_0)_{+} v^* - (p - \ep)_{+} \bigr\| < \ep.
\]
Theorem 2.13 of~\cite{ArPrTm} provides $x \in A$ such that
\[
\| x \| \leq 1
\andeqn
x v (a - \dt_0)_{+} v^* x^* = (p - 2 \ep)_{+}.
\]
Set
\[
y = (1 - 2 \ep)^{- 1/2} x v [(a - \dt_0)_{+}]^{1/2}.
\]
Since $(p - 2 \ep)_{+} = (1 - 2 \ep) p$, we have $y y^* = p$.
Therefore
\[
q = y^* y
  = (1 - 2 \ep)^{-1} [(a - \dt_0)_{+}]^{1/2} v^* x^*
        x v [(a - \dt_0)_{+}]^{1/2}
\]
is a \pj.
Using $\| x \| \leq 1$ and setting $\af = (1 - 2 \ep)^{-1} \| v \|^2$,
we get
\[
q \leq (1 - 2 \ep)^{-1} \| v \|^2 (a - \dt_0)_{+}
  \leq \af a.
\]
Also, using $\dt = \dt_0 / 2$ in the definition of~$g$,
we get $g (a) (a - \dt_0)_{+} = (a - \dt_0)_{+}$,
so $g (a) q = q$.
\end{proof}

The next lemma is a more precise version
of Corollary~2.4 of~\cite{EllNiu2},
which says that $\rc (C (X)) \in \Nz \cup\{ \I \}$.
We give a full proof since it is easy to be off by~$1$.

\begin{lem}\label{L_1725_Dfn_rcCX}
Let $X$ be a \chs.
Then $\rc (C (X))$ is the least $r \in \Nz \cup\{ \I \}$
such that whenever $n \in \N$ and $a, b \in C (X, M_n)_{+}$ satisfy
\[
\rank (a (x)) + r + 1 \leq \rank (b (x))
\]
for all $x \in X$, then $a \precsim b$.
\end{lem}

\begin{proof}
Let $r$ be the integer in the statement.
Lemma~2.3 of~\cite{EllNiu2} implies that
$\rc (C (X))$ is the infimum of all $\rh > 0$
such that whenever $n \in \N$ and $a, b \in C (X, M_n)_{+}$ satisfy
\[
\rank (a (x)) + \rh < \rank (b (x))
\]
for all $x \in X$, then $a \precsim b$.
It therefore suffices to show that this condition holds if $\rh \geq r$
and fails if $\rh < r$.

For the first, suppose $\rh \geq r$
and $\rank (a (x)) + \rh < \rank (b (x))$ for all $x \in X$.
Since $r$, $\rank (a (x))$, and $\rank (b (x))$ are integers,
we must have $\rank (a (x)) + r + 1 \leq \rank (b (x))$ for all $x \in X$.
So $a \precsim b$ by the definition of~$r$.

Now suppose $\rh < r$.
By the definition of~$r$, there are $a, b \in C (X, M_n)_{+}$
such that $\rank (a (x)) + r \leq \rank (b (x))$ for all $x \in X$
but $a \not\precsim b$.
Then $\rank (a (x)) + \rh < \rank (b (x))$ for all $x \in X$,
and $a \not\precsim b$.
\end{proof}

We state the next lemma separately for convenient reference.

\begin{lem}\label{L_2821_PartE}
Let $X$ be a \chs, let $X_0 \S X$ be closed, let $d \in \N$,
and let $\xi_0 \colon X_0 \to S^d$ be \ct.
Then there are a closed subset $X_1 \S X$ with $X_0 \S \sint (X_1)$
and a \cfn{} $\xi_1 \colon X_1 \to S^d$ such that $\xi_1 |_{X_0} = \xi_0$.
\end{lem}

\begin{proof}
Identify $S^d$ with the unit sphere in $\R^{d + 1}$.
The Tietze Extension Theorem provides a \cfn{} $\et \colon X \to \R^{d + 1}$
which extends~$\xi_0$.
Set $X_1 = \bigl\{ x \in X \colon \| \et (x) \| \geq \frac{1}{2} \bigr\}$
and $\xi_1 (x) = \| \et (x) \|^{-1} \et (x)$.
\end{proof}

\begin{prp}\label{P_1725_ExtPr1}
Let $X$ be a \chs.
Let $r \in \Nz$ satisfy $r \geq \rc (C (X))$.
Then for every closed subset $X_0 \S X$
and every \cfn{} $\xi_0 \colon X_0 \to S^{2 r + 5}$,
there exist \cfn{s} $\xi \colon X \to S^{2 r + 5}$
and $\zt \colon X_0 \to S^1$ such that,
identifying $S^{2 r + 5}$ with the unit sphere in $\C^{r + 3}$,
for all $x \in X_0$ we have $\xi (x) = \zt (x) \xi_0 (x)$.
\end{prp}

\begin{proof}
We can assume $X_0 \neq \E$.
Let $X_1 \S X$ and $\xi_1 \colon X_1 \to S^{2 r + 5}$
be as in Lemma~\ref{L_2821_PartE}.

We use the usual scalar product on $\C^{r + 3}$.
For $x \in X_1$ define a rank one \pj{}
$p_1 (x) \in M_{r + 3} = L (\C^{r + 3})$
by $p_1 (x) \et = \langle \et, \, \xi_1 (x) \rangle \xi_1 (x)$
for all $\et \in \C^{r + 3}$.
Then $p_1$ is a \pj{} in $C (X_1, M_{r + 3})$.
Set $p_0 = p_1 |_{X_0}$.
Fix $x_0 \in X_0$, and let $e \in C (X, M_{r + 3})$ be the
constant \pj{} with value $p_1 (x_0)$.
Set $\et_1 = \xi_1 (x_0)$.

We claim that $p_1$ is trivial.
To see this, define $s \in C (X_1, M_{r + 3})$ by
$s (x) \et = \langle \et, \, \et_1 \rangle \xi_1 (x)$
for $x \in X_1$ and $\et \in \C^{r + 3}$.
Then check that
$s (x)^* \et = \langle \et, \, \xi_1 (x) \rangle \et_1$
for $x \in X_1$ and $\et \in \C^{r + 3}$, that $s s^* = p_1$,
and that $s^* s = e |_{X_1}$.

Choose a \cfn{} $f \colon X \to [0, 1]$ such that $f (x) = 1$
for all $x \in X_0$ and $\supp (f) \S \sint (X_1)$.
Define $b \in C (X, M_{r + 3})$ by
\[
b (x) = \begin{cases}
    1 - f (x) p_1 (x) & \hspace*{1em} x \in X_1
        \\
   1 & \hspace*{1em} x \not\in X_1.
\end{cases}
\]

For $x \in X$, we have
\[
\rank (b (x)) \geq r + 2 = r + 1 + \rank (e (x)).
\]
Therefore $e \precsim_{C (X)} b$ by Lemma~\ref{L_1725_Dfn_rcCX}.
Lemma~\ref{L_1725_Cut} gives a \pj{} $q \in C (X, M_{r + 3})$
and $\dt > 0$
such that $q$ is \mvnt{} to~$e$ and, with $g$ as in the statement
of that lemma, we have $g (b) q = q$.
In particular, $q$ is a rank one trivial \pj{}
with $q |_{X_0} \leq 1 - p_0$.
Since $p_0$ is a rank one trivial \pj{} on~$X_0$,
Lemma~\ref{L_1723_Switch} implies that there is a rank one trivial \pj{}
$p \in C (X, M_{r + 3})$ such that $p |_{X_0} = p_0$.
By triviality, there is $t \in C (X, M_{r + 3})$ such that
$t t^* = p$ and $t^* t = e$.
Define $\xi \colon X \to \C$ by $\xi (x) = t (x) \et_1$.
Then $\| \xi (x) \|_2 = 1$ for all $x \in X$,
so $\xi$ is a function from $X$ to $S^{2 r + 5}$.
Moreover, for $x \in X_0$, we have $p (x) = p_0 (x)$,
so $\xi (x) \in \C \xi_0 (x)$,
and $\zt (x) = \langle \xi (x), \xi_0 (x) \rangle$ is a \cfn{}
with $| \zt (x) | = 1$ and such that $\xi (x) = \zt (x) \xi_0 (x)$.
\end{proof}

\begin{rmk}\label{R_1726_Ext_pjs}
The conclusion of Proposition~\ref{P_1725_ExtPr1} can be restated
in terms of projections, as follows.
For every closed subset $X_0 \S X$
and every rank one trivial \pj{} $p_0 \in C (X_0, M_{r + 3})$,
there exists a rank one trivial \pj{} $p \in C (X, M_{r + 3})$
such that $p (x) = p_0 (x)$ for all $x \in X_0$.
\end{rmk}

\section{From extensions of projections to radius of
 comparison}\label{Sec_1801_Sec__2}

\indent
The following result and its proof were suggested by Thomas Nikolaus.

\begin{prp}\label{T_1801_Htpy}
Let $d \in \N$ be even.
Identify $S^{2 d - 1}$ with the unit sphere in~$\C^d$.
Define $m, p \colon S^1 \times S^{2 d - 1} \to S^{2 d - 1}$
by $m (\ld, \xi) = \ld \xi$ and $p (\ld, \xi) = \xi$
for $\ld \in S^1$ and $\xi \in S^{2 d - 1}$.
Then $m$ and $p$ are homotopic.
\end{prp}

\begin{proof}
Let ${\mathbb{H}}$ be the quaternions.
Recall that ${\mathbb{H}}$, in its usual norm,
is isometrically isomorphic, as a complex normed vector space, to $\C^2$,
and that the elements $1, i, j, k \in {\mathbb{H}}$
form an orthonormal basis.
Using this identification, choose an isometric isomorphism of $\C^d$
with ${\mathbb{H}}^{d / 2}$.
This gives a multiplication map ${\mathbb{H}} \times \C^d \to \C^d$.
We identify the quaternions of norm~$1$ with $S^3$.
This gives an inclusion $g \colon S^1 \to S^3$.
Since multiplication by a quaternion of norm~$1$ is isometric,
the multiplication map above restricts to
$n \colon S^3 \times S^{2 d - 1} \to S^{2 d - 1}$,
and $n \circ (g \times \id_{S^{2 d - 1}}) = m$.

Since $\pi_1 (S^3)$ is trivial,
$g$~is homotopic to the composition
of the obvious map $q \colon S^1 \to \{ 1 \}$
and the inclusion $t$ of $\{ 1 \}$ in $S^3$,
that is, $g \simeq t \circ q$.
Therefore, using a direct computation of the map at the last step,
\[
m = n \circ (g \times \id_{S^{2 d - 1}})
  \simeq n \circ (t \times \id_{S^{2 d - 1}})
        \circ (q \times \id_{S^{2 d - 1}})
  = p.
\]
This completes the proof.
\end{proof}

Proposition~\ref{T_1801_Htpy} may well to be false when $d$ is odd.

\begin{prp}\label{T_1801_Bound}
Let $X$ be a \chs.
Let $r \in \Nz$ be odd and satisfy $\dim (X) \geq 2 r + 6$.
Then $\rc (C (X)) \geq r + 1$.
\end{prp}

\begin{proof}
It is equivalent to prove that if
$r \in \Nz$ is odd and satisfies $\rc (C (X)) \leq r$,
then $\dim (X) \leq 2 r + 5$.
We use the criterion of Theorem~\ref{T_1801_Dim_exts}.
Thus, let $X_0 \S X$ be closed,
and let $\xi_0 \colon X_0 \to S^{2 r + 5}$ be \ct.
Let $X_1 \S X$ and the extension $\xi_1 \colon X_1 \to S^{2 r + 5}$
of $\xi_0$ be as in Lemma~\ref{L_2821_PartE}.

Apply Proposition~\ref{P_1725_ExtPr1} with $X_1$ in place of $X_0$
and $\xi_1$ in place of $\xi_0$, getting
\cfn{s} $\et \colon X \to S^{2 r + 5}$ (called $\xi$ there)
and $\zt \colon X_0 \to S^1$ such that,
identifying $S^{2 r + 5}$ with the unit sphere in $\C^{r + 3}$,
for all $x \in X_0$ we have $\et (x) = \zt (x) \xi_0 (x)$.
Also choose a \cfn{} $f \colon X \to [0, 1]$ such that
$f (x) = 1$ for all $x \in X_0$ and $\supp (f) \S \sint (X_1)$.

By Proposition~\ref{T_1801_Htpy} (with $n = r + 3$),
there is a \cfn{}
\[
H \colon [0, 1] \times S^1 \times S^{2 r + 5} \to S^{2 r + 5}
\]
such that for all $y \in S^{2 r + 5}$ and $\ld \in S^1$ we have
\[
H (0, \ld, y) = y
\andeqn
H (1, \ld, y) = \ld y.
\]
Now for $x \in X$ define
\[
\xi (x)
 = \begin{cases}
   H \bigl( f (x), \, \zt (x)^{-1}, \, \et (x) \bigr)
         & \hspace*{1em} x \in X_1
       \\
   \et (x) & \hspace*{1em} x \not\in X_1.
\end{cases}
\]
For $x \in \partial X_1$ we have $f (x) = 0$,
so $H \bigl( f (x), \, \zt (x)^{-1}, \, \et (x) \bigr) = \et (x)$.
Therefore $\xi$ is a \cfn{} from $X$ to $S^{2 r + 5}$
whose restriction to $X_0$ is~$\xi_0$.
\end{proof}

If we take apart the parity conditions in Proposition~\ref{T_1801_Bound},
we get the following estimates.

\begin{thm}\label{T_2705_Ests}
Let $X$ be a \chs.
\begin{enumerate}
\item\label{I_T_2705_Inf}
If $\dim (X) = \I$, then $\rc (C (X)) = \I$.
\item\label{I_T_2705_0}
If $\dim (X) \equiv 0 \pmod{4}$,
then $\rc (C (X)) \geq \frac{1}{2} \bigl( \dim (X) - 4 \bigr)$.
\item\label{I_T_2705_1}
If $\dim (X) \equiv 1 \pmod{4}$,
then $\rc (C (X)) \geq \frac{1}{2} \bigl( \dim (X) - 5 \bigr)$.
\item\label{I_T_2705_2}
If $\dim (X) \equiv 2 \pmod{4}$,
then $\rc (C (X)) \geq \frac{1}{2} \bigl( \dim (X) - 6 \bigr)$.
\item\label{I_T_2705_3}
If $\dim (X) \equiv 3 \pmod{4}$,
then $\rc (C (X)) \geq \frac{1}{2} \bigl( \dim (X) - 7 \bigr)$.
\end{enumerate}
\end{thm}

\begin{proof}
Part~(\ref{I_T_2705_Inf}) is clear from Proposition~\ref{T_1801_Bound}.
For~(\ref{I_T_2705_0}),
apply Proposition~\ref{T_1801_Bound} with $r = \frac{1}{2} \dim (X) - 3$.
For~(\ref{I_T_2705_1}) use $r = \frac{1}{2} \dim (X) - \frac{7}{2}$,
for~(\ref{I_T_2705_2}) use $r = \frac{1}{2} \dim (X) - 4$,
and for~(\ref{I_T_2705_3}) use $r = \frac{1}{2} \dim (X) - \frac{9}{2}$.
\end{proof}

The best general statement is the following corollary.

\begin{cor}\label{C_2705_Cor}
Let $X$ be a \chs.
Then $\rc (C (X)) \geq \frac{1}{2} \bigl( \dim (X) - 7 \bigr)$.
\end{cor}

\end{document}